\documentclass[12pt]{article}
\usepackage{amssymb}
\usepackage{amsmath}
\usepackage{amsthm}
\usepackage{authblk}
\usepackage[title]{appendix}
\makeatother
\usepackage{color} 
\usepackage{enumerate}
\usepackage{geometry,pdfpages}
\usepackage{graphicx}
\usepackage{tikz}
\usetikzlibrary{arrows.meta, positioning}
\usepackage[hyphens]{url}
\def\r0{\mathcal{R}_0}

\newtheorem{theorem}{Theorem}[section]

\theoremstyle{definition}
\newtheorem*{note*}{Remark}
\begin{document}

\title{When ideas go viral - complex bifurcations \\ in a two-stage transmission model}

\author{J.~HEIDECKE and M.~V. BARBAROSSA}

\affil{Frankfurt Institute for Advanced Studies, \\
	Ruth-Moufang-Straße 1 \\ 
	60438 Frankfurt, Germany\\ 
	heidecke@fias.uni-frankfurt.de}
\date{} 

\maketitle

\begin{abstract}
	We consider the qualitative behavior of a mathematical model for transmission
	dynamics with two nonlinear stages of contagion. The proposed model is inspired by phenomena occurring in epidemiology (spread
	of infectious diseases) or social dynamics (spread of opinions, behaviors, ideas), and described by a compartmental approach.
	Upon contact with a promoter (contagious individual), a naive (susceptible) person can either become promoter himself or become \textit{weakened}, hence more vulnerable. Weakened individuals become contagious when they experience a second contact with members of the promoter group. After a certain time in the contagious compartment, individuals become inactive (are insusceptible and cannot spread) and are removed from the chain of transmission. We combine this two-stage contagion process with renewal of the naive population, modeled by means of transitions from the weakened or the inactive status to the susceptible compartment. This leads to rich dynamics, showing for instance
	coexistence and bistability of equilibria and periodic orbits. Properties of (nontrivial) equilibria are studied analytically. In addition, a numerical investigation of the parameter space reveals numerous bifurcations, showing that the dynamics of such a system can be more complex than those of classical epidemiological ODE models.
\end{abstract}
\section{Introduction}
Social contagion is the spread of behaviors or attitudes through (physical or virtual) groups of people~\cite{dicofpsych}. From a mathematical point of view, modeling social contagion in large communities is very similar to modeling the transmission of an infectious disease in a population. Hence, it seems natural that methods from the field of mathematical epidemiology, such as compartmental models~\cite{sook,goffman,bettencourt,jin,woo,wang2016,kat}, are used to model social contagion phenomena. In certain cases social contagion and disease spread even have to be considered coupled to one another, as when a group in a social network criticizes vaccination~\cite{bauch}. Despite the analogies, social contagion differs from biological contagion in various aspects. For example, intellectual epidemics could be advantageous~\cite{goffman}, ideas do not require interpersonal contact to spread~\cite{hill}, or people might be asked to choose between opposite opinions~\cite{wang2016}. Thus, to mathematically describe social contagion processes, models from theoretical epidemiology might require adaptation to the specific context.\\
\ \\
The classical SEIR (susceptible-exposed-infective-recovered) model in mathematical epidemiology describes the transmission of an infectious disease in a population~\cite{martch}. When a susceptible individual comes in contact with an infective one, there is a certain probability that contagion occurs and the susceptible moves to the exposed compartment. After a latent period exposed individuals become infectious themselves and can infect others. Once the infectious period is over, the individual recovers, cannot transmit the disease to others anymore and becomes immune. Waning of immunity, i.e., transitions from the recovered to the susceptible compartment, is possible for certain diseases~\cite{martch,mvbrost}.\\
\ \\
In the context of social contagion, the spreading of a specific behavior or opinion in a population could be described as the transition of individuals from the "naive" (susceptible) status to the "promoter" (infectious) one. This transition might require several steps and depend on repeated exposure to promoters~\cite{kat,roger}. Therefore, we classify individuals as:
\begin{itemize}\setlength\itemsep{1em}
    \item  \textbf{naive/susceptible} ($S$), those who have not yet been exposed to the considered behavior/opinion,
    \item \textbf{weakened} ($W$), those who came in contact with the considered behavior/opinion, but are not yet spreading to others
    \item \textbf{promoters/infectious} ($I$), those who have embraced the considered behavior/opinion and are able to transmit it to others
    \item \textbf{inactive/resistant} ($R$), those who have been sharing the considered behavior/opinion earlier but are now neither transmitting to others nor can be re-exposed. 
\end{itemize}
\noindent In contrast to the classical SEIR approach, transition from the exposed/weakened stage to the promoter/infected stage depends on contacts with infectives/promoters. 
Promoters and susceptibles make contact sufficient to transmit the opinion/behavior at rate $\beta$. We assume that upon such a contact the susceptibles have a certain probability $\rho\in(0,1)$ to enter the $I$-compartment directly, becoming a promoter themselves (\textit{perfect contact}). With probability $1-\rho$ susceptible individuals enter the $W$-compartment (\textit{imperfect contact}). When a $W$-individual comes in contact with promoters, the weakened enters the $I$-compartment at rate $\beta_2\geq \beta$. In this sense, an individual in the $W$-compartment is more vulnerable to the opinion/behavior than an individual in the $S$-compartment. Over time promoters might reject the considered opinion/behavior. Thus, we assume that promoters leave the promoting class at rate $\alpha$ and become inactive/resistant. As we assume that the inactivity/resistance of a $R$-individual towards the considered opinion/behavior might wane over time, we allow transition from $R$ back to $S$ at rate $\eta$. Moreover, we assume that the infection can fade away in weakened individuals, whereby transitions from $W$ to $S$ are occuring at rate $\gamma$. We assume that $\gamma\geq\alpha$, suggesting that the average time an individual is able to promote a certain opinion is not shorter than his exposure time.
\begin{figure}[t]
	\centering
	\includegraphics[width=0.85\textwidth]{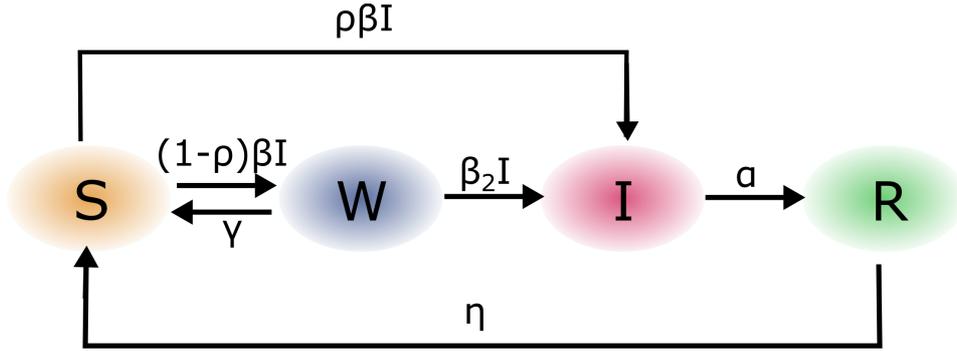}
	\caption{Flowchart of the two-stage contagion model~\eqref{eq:model} with waning of immunity and fading of infection in weakened individuals.}
	\label{fig:Figure0}
\end{figure}
A model sketch is given in Fig.~\ref{fig:Figure0} and the corresponding differential equations system is
\begin{equation}
\begin{aligned}
S'&=-\beta IS+\gamma W+\eta R\\
W'&=(1-\rho)\beta IS-\beta_2 IW-\gamma W\\
I'&=\beta_2 IW+\rho\beta IS-\alpha I\\
R'&=\alpha I-\eta R.
\label{eq:model}
\end{aligned}
\end{equation} 
Limit cases of this system lead, on the one hand, to the standard SIRS model (cf.~\cite{martch}) if $\rho=1$ and the $W$-compartment is empty at the beginning of observations. On the other hand, if the transition from $W$ to $I$ would be a  linear one and $\gamma=\rho=0$, the model would be equivalent to the classical SEIRS model (cf.~\cite{martch}).\\
\ \\
{Observing that the total population $N=S+W+I+R$ does not vary over time, we set $N\equiv 1$. Clearly, system~\eqref{eq:model} has a unique non-negative global solution for every choice of non-negative initial values, and the set $\{(S,I,W,R)\in\mathbb{R}^4_{\geq 0}\,|\,S+W+I+R=1\}$ is forward invariant. For analytical simplicity, we restrict to the case $\beta_2=\beta$. Using the conservation relation $W=1-S-I-R$, we consider the reduced system}
\begin{equation}
\begin{aligned}
S'&=-\beta IS+\gamma(1-S-I-R)+\eta R\\
I'&=\rho\beta IS-\alpha I +\beta I(1-S-I-R)\\
R'&=\alpha I-\eta R.
\label{eq:reducedmodel}
\end{aligned}
\end{equation}
A similar compartmental model for two-stage contagion was previously proposed by Guy Katriel~\cite{kat}. System~\eqref{eq:reducedmodel} differs from Katriel's work in two aspects. First we include waning immunity and fading of infection, that is transitions from $R$, respectively $W$, to the susceptible compartment. Second, we do not consider population demography (births/deaths). Such differences lead to major analytical challenges with respect to Katriel's study~\cite{kat}. In the rest of this work we study the qualitative properties of system~\eqref{eq:reducedmodel} by means of analytical and numerical methods. 

\section{Existence and local stability of equilibria}
\label{sec:results}
To understand the long-term behavior of the system we first investigate its equilibria. The criteria on the stability of the disease-free equilibrium, where the $I$ compartment is empty, are commonly related to the so-called \textit{basic reproduction number}, $\mathcal{R}_0$. We remark that the following results are derived assuming that $\gamma \geq \alpha$ holds.
\begin{theorem}
\begin{enumerate}[a)]
    \item System \eqref{eq:reducedmodel} has a unique disease-free (DFE) equilibrium $\mathcal{E}_0:=(1,0,0)$, which exists for any choice of $\beta,\,\gamma,\,\eta,\,\alpha>0$, $\rho \in (0,1)$.
    \item The DFE $\mathcal{E}_0$ is locally asymptotically stable (shortly, LAS) if $\mathcal{R}_0:=\frac{\rho\beta}{\alpha}<1$, and unstable if $\mathcal{R}_0>1$.
\end{enumerate}
\label{thm:DFEstab}	
\end{theorem}
\begin{proof}
\textit{a)}: Equilibria of system~\eqref{eq:reducedmodel} are determined setting the right-hand side of the system equal to zero,
\begin{equation}
    \begin{aligned}
    0=&-\beta I^*S^*+\gamma(1-S^*-I^*-R^*)+\eta R^*\\
	0=&I^*\left(\rho\beta S^*-\alpha  +\beta (1-S^*-I^*-R^*)\right)\\
	0=&\alpha I^*-\eta R^*.
    \label{eq:equilibria1}
    \end{aligned}
\end{equation}
	If $I^*=0$ (DFE condition), then the third relation in~\eqref{eq:equilibria1} implies that $R^*=0$ and the first relation in~\eqref{eq:equilibria1} implies that $S^*=1$. Thus, $\mathcal{E}_0$ is the unique DFE of the system. 
	\textit{b)}: For studying the local stability of the DFE we linearize about $\mathcal{E}_0$. The Jacobian matrix of~\eqref{eq:reducedmodel} is
	\begin{equation}
	J=
	\left( {\begin{array}{ccc}
		-\beta I-\gamma & -\beta S-\gamma & \eta-\gamma\\
		\left(\rho-1\right)\beta I\,\, &\beta\left( \left(\rho-1\right)S+1-2 I - R\right)-\alpha\,\, & -\beta I \\
		0 & \alpha & -\eta \\	
		\end{array} } \right).
	    \label{Jacobian}
	\end{equation}
Evaluation of $J$ at $\mathcal{E}_0$ yields the eigenvalues $\lambda_1=-\gamma<0$, $\lambda_2=\rho\beta-\alpha$ and $\lambda_3=-\eta<0$. The condition $\lambda_2<0 \iff \r0<1$ guarantees local stability of the DFE.
\end{proof}
The above expression for $\mathcal{R}_0$ can be interpreted as the number of secondary infections that a single promoter introduced into a completely susceptible population produced through {perfect contacts} ($\rho\beta$) over the duration of its promoting period ($1/\alpha$).
The secondary infections produced through imperfect contacts require two nonlinear transitions over the $W$-compartment, which is empty at $\mathcal{E}_0$, hence do not contribute to determining the local stability of the DFE.\\\\
Theorem~\ref{thm:DFEstab} provides conditions for the local stability of the DFE. What exactly happens at the bifurcation value $\r0=1$ can be determined with the help of Theorem~\ref{theo:chavez}, first introduced by Castillo-Chavez and Song~\cite{chavez2004}. For convenience of notation we define  $\kappa:=\left(1+\alpha/\eta\right)$.
\begin{theorem}\label{theo:bifdir}
	Let $\rho^*:=\frac{-\alpha+\sqrt{\alpha\gamma\kappa}}{\gamma\kappa-\alpha}$. If $\rho>\rho^*$, then a transcritical bifurcation of forward type occurs at $\mathcal{R}_0=1$. If $\rho<\rho^*$, then a transcritical bifurcation of backward type occurs at $\mathcal{R}_0=1$.
\end{theorem}
\begin{proof}
	We use Theorem \ref{theo:chavez} and take $\beta$ as our bifurcation parameter, with $\beta^*:=\frac{\alpha}{\rho}$ corresponding to $\mathcal{R}_0=1$. The Jacobian~\eqref{Jacobian} of system~\eqref{eq:reducedmodel} evaluated at $\mathcal{E}_0$ for $\beta=\beta^*$, $
	\mathcal{A}:=J|_{\mathcal{E}_0, \beta^*}$, has a simple zero eigenvalue and two eigenvalues with negative real part. To compute a right eigenvector $w$ of $\mathcal{A}$ corresponding to the zero eigenvalue, solve $\mathcal{A}w=0$, which is an underdetermined system. We fix $w_2=1$ and obtain $w_3=\frac{\alpha}{\eta}w_2=\frac{\alpha}{\eta}$ and $w_1=-\frac{\alpha}{\rho\gamma}-1+\frac{\alpha}{\gamma}-\frac{\alpha}{\eta}$. Analogously, to find a left eigenvector, we solve the system $v\mathcal{A}=0$, obtaining $v=(0,1,0)$.\\ 
	\ \\
	Next, we denote the vector field of system~\eqref{eq:reducedmodel} by $f=(f_1,f_2,f_3)$ and we calculate second order partial derivatives to verify the conditions of Theorem~\ref{theo:chavez}.
	\linebreak
	As $v_1=0=v_3$ the derivatives of $f_1$ and $f_3$ are not needed. All second order derivatives of $f_2$ are zero, except for
	\begin{align*}
	\frac{\partial^2f_2}{\partial^2 I}&=-2\beta^*,&
	\frac{\partial^2f_2}{\partial I\partial R}&=-\beta^*,\\
	\frac{\partial^2f_2}{\partial S\partial I}&=\rho\beta^*-\beta^*,&
	\frac{\partial^2f_2}{\partial I\partial \beta^*}&=\rho.
	\end{align*}
	Therefore, the quantities $a$ and $b$ in Theorem~\ref{theo:chavez} are given by
	\begin{equation*}
	b=\sum_{i=1}^3 v_2w_i\frac{\partial^2 f_2}{\partial x_i\partial \beta}(0,0)=\rho v_2 w_2=\rho>0,
	\end{equation*}
	and 
	\begin{align*}
	a=&\sum_{i,j=1}^{3}v_2w_iw_j\frac{\partial^2 f_2}{\partial x_i\partial x_j}(0,0)\\=&-2\frac{\alpha}{\rho}v_2w_2^2-2\frac{\alpha}{\rho}v_2w_2w_3+2\left(\alpha-\frac{\alpha}{\rho}\right)v_2w_1w_2\notag\\
	=&-2\frac{\alpha}{\rho}-2\frac{\alpha^2}{\rho\eta}+2\left(\alpha-\frac{\alpha}{\rho}\right)\left(-\frac{\alpha}{\rho\gamma}-1+\frac{\alpha}{\gamma}-\frac{\alpha}{\eta}\right),\notag
	\end{align*}
	where $x=(x_1,x_2,x_3)=(S,I,R)$. The condition $a>0$ is equivalent to
	\begin{align}
	\psi(\rho):=\rho^2\left(\gamma\kappa-\alpha\right)+2\rho\alpha-\alpha<0.\label{parabola}
	\end{align}
	Because of the assumption $\gamma\geq\alpha$, we have $\gamma\kappa=\gamma(1+\frac{\alpha}{\eta})>\alpha$. The open up parabola $\psi(\rho)$ has negative intercept and vertex on the left half-plane. Thus, condition~\eqref{parabola} is fulfilled only if $0<\rho<\rho^*$, where 
	\begin{equation*}
	\rho^*=\frac{-\alpha+\sqrt{\alpha\gamma\kappa}}{\gamma\kappa-\alpha}
	\end{equation*}
is the positive zero of $\psi(\rho)$.
\end{proof} 
Endemic equilibria of~\eqref{eq:reducedmodel} are determined by solving~\eqref{eq:equilibria1} for $I^*\neq 0$. From the third and the first relation in~\eqref{eq:equilibria1} we get $$
R^*=\frac{\alpha}{\eta} I^* \; \mbox{ and } \; S^* = \frac{\gamma(1-\kappa I^*)+\alpha I^*}{\beta I^*+\gamma}.$$
Note here that $1-I^*-\frac{\alpha}{\eta}I^*=1-I^*-R^*\geq0$, hence $S^*\geq0$ if $I^*\geq0$. The equilibria condition thus reduces to the quadratic equation  
\begin{eqnarray}
0=a{I^*}^2+bI^*+c=:\Phi(I^*),\label{eq:quadratic}
\end{eqnarray}
with $a:=\beta\kappa>0$, $b:=\alpha\left(2-\rho\right)+\rho\gamma\kappa-\beta$ and $c:=\gamma\left(\frac{\alpha}{\beta}-\rho\right)$.
\pagebreak\vfil
Solutions to~\eqref{eq:quadratic} are given by 
\begin{equation*}
I^*_\pm=\frac{1}{2\beta\kappa}\left[\beta-\rho\gamma\kappa-\alpha\left(2-\rho\right)\pm\sqrt{\Delta}\right],
\end{equation*}
where $\Delta$ is the discriminant of~\eqref{eq:quadratic}. Hence, system~\eqref{eq:reducedmodel} has at most two endemic equilibria $\mathcal{E}_\pm:=\left(S_\pm^*,I_\pm^*,R_\pm^*\right)$,
whereby $I^*_\pm$ are the positive roots of the open up parabola $\Phi(I^*)$. 
We write the coefficients $b$ and $c$ of $\Phi(I^*)$ as functions of $\beta$, 
\begin{align*}
	b(\beta):=&\alpha(2-\rho)+\rho\gamma\kappa-\beta\\
	c(\beta):=&\gamma\left(\frac{\alpha}{\beta}-\rho\right).
\end{align*}
Note that $b(\beta)$ is a strictly decreasing linear function with zero at
\begin{align*}
	\beta_b:=\alpha(2-\rho)+\rho\gamma\kappa
\end{align*}
and $c(\beta)$ is a strictly decreasing curve with zero at
\begin{align*}
	\beta_c:=\frac{\alpha}{\rho}.
\end{align*}
Observe further that 
\begin{equation*}
c(\beta)
\left\{
\begin{matrix}
>0 \iff \r0 <1,\\
=0 \iff \r0 =1,\\
<0 \iff \r0 >1.
\end{matrix}
\right.
\end{equation*}
Theorem~\ref{theo:bifdir} suggest to consider the cases $\rho^*<\rho$ and $\rho^*>\rho$ separately.
\begin{theorem}\label{Theorem1}
	If $\rho^*<\rho<1$, then:
	\begin{itemize}
		\item For $\mathcal{R}_0\leq 1$ there are no endemic equilibria.
		\item For $\mathcal{R}_0>1$ there is a unique endemic equilibrium $\mathcal{E}_+$.
	\end{itemize}
	At $\mathcal{R}_0=1$ a transcritical bifurcation of forward type occurs, and a branch of stable endemic equilibria $\mathcal{E}_+$ emerges from $\mathcal{E}_0$. 
\end{theorem}
\begin{proof}
Observe that
	\begin{equation}
	    \beta_b> \beta_c \,\iff\, 2\rho\alpha-\alpha\rho^2+\rho^2\gamma\kappa> \alpha\;\iff\, \psi(\rho)>0
	    \,\underbrace{\iff}_{\mbox{Thm.~\ref{theo:bifdir}}}\, \rho^*<\rho\label{eq:CondOnRho}.
	\end{equation}
Thus, if $\rho^*<\rho$ we have that the zero of $b(\beta)$ lies on the right of the zero of $c(\beta)$. In other words, if $c(\beta)\geq 0$ (that is $\r0\leq 1$) then necessarily $b(\beta)> 0$. This means that if the open up parabola~$\Phi(I^*)$ has a positive y-intercept, then it also has vertex on the left half-plane, hence no positive root $I^*$. If $c(\beta)<0$ (that is $\r0>1$), $\Phi(I^*)$ has a negative y-intercept, thus only its larger root $I^*_+$ is positive. The rest follows from Theorem~\ref{theo:bifdir}.
\end{proof}
Next, define \begin{equation}
    \mathcal{R}_C:=\rho(2-\rho)-\rho^2\frac{\gamma\kappa}{\alpha}+2\frac{\rho(1-\rho)}{\alpha}\sqrt{\alpha\gamma\kappa}
\label{def:Rc}
\end{equation}
and consider the case $\rho^*>\rho$.
\pagebreak\vfill
\begin{theorem}\label{Theorem2}
	If $0<\rho<\rho^*$, then the following results hold:
	\begin{enumerate}[a)]
		\item If $\mathcal{R}_0<\mathcal{R}_C$, there are no endemic equilibria.
		\item If $\mathcal{R}_C\leq\mathcal{R}_0<1$, there are two endemic equilibria $\mathcal{E}_{\pm}$, which coincide if $\mathcal{R}_C=\mathcal{R}_0$.
		\item If $1\leq\mathcal{R}_0$, there is a unique endemic equilibrium $\mathcal{E}_+$.
	\end{enumerate}
	At $\mathcal{R}_0=1$ a transcritical bifurcation of backward type occurs, and a branch of unstable endemic equilibria $\mathcal{E}_-$ emerges from $\mathcal{E}_0$.
\end{theorem}
\begin{proof}
From $\rho<\rho^*$ it follows from~\eqref{eq:CondOnRho} that $\beta_b <\beta_c$. Hence if $c(\beta)\leq 0$ (that is $\r0\geq 1$) then necessarily $b(\beta)<0$.
This means that if the open up parabola~$\Phi(I^*)$ has a non-positive y-intercept, it has vertex on the right half-plane as well, and hence a unique positive root $I_+^*$. This proofs statement \textit{c)}.
For two positive roots $I_\pm^*$ the conditions $c(\beta)> 0$, $b(\beta)< 0$ and $\Delta(\beta)\geq 0$, where 
\begin{equation*}
\begin{aligned}
	\Delta(\beta)=&\beta^2+\beta{\left(2\rho\alpha+2\rho\gamma\kappa-4\alpha\right)}\\
	&\; +\underbrace{4\alpha^2\left(1-\rho\right)+\rho^2\alpha^2+\rho^2\gamma^2\kappa^2+4\rho\alpha\gamma\kappa-2\rho^2\alpha\gamma\kappa-4\alpha\gamma\kappa}_{:=c_\Delta},
	\end{aligned}
\end{equation*} 
are necessary. For $\Delta(\beta)=0$ the two roots coincide. The discriminant $\Delta(\beta)$ as a function of $\beta$ is an open up parabola itself. Short computation shows that 
$$ c_\Delta<0\; \iff\; \rho<2\rho^*,$$ 
hence $c_\Delta$ is always negative under the assumptions of the theorem. Thus  $\Delta(\beta)\geq0$ for $\beta$ larger than the positive zero of the discriminant-parabola, that is  
\begin{align*}
	\beta&\geq\alpha(2-\rho)-\rho\gamma\kappa+2(1-\rho)\sqrt{\alpha\gamma\kappa}=:\beta_\Delta,
\end{align*}
and $\Delta(\beta)=0\iff \beta=\beta_\Delta$. Observe that 
\begin{equation}
\begin{aligned}
	\beta_b<\beta_\Delta
	&\iff\;\rho^2\gamma^2\kappa^2<(1-\rho)^2\alpha\gamma\kappa\\&\iff\;\rho^2(\gamma\kappa-\alpha)+\rho2\alpha-\alpha<0\\&\iff\;\rho<\rho^*.
\end{aligned}
	\label{rel_betaBbetaD}
\end{equation}
Further
\begin{align*}
	\beta_\Delta<\beta_c&\iff\alpha(2-\rho)-\rho\gamma\kappa+2(1-\rho)\sqrt{\alpha\gamma\kappa}<\frac{\alpha}{\rho}\notag\\&\iff
	\underbrace{\rho^2\left(-\alpha-\gamma\kappa-2\sqrt{\alpha\gamma\kappa}\right)+\rho2\left(\alpha+\sqrt{\alpha\gamma\kappa}\right)-\alpha}_{=:\chi(\rho)}<0.
\end{align*}
Note that $\chi(\rho)$ is an open down parabola with vertex on the right half-plane, negative y-intercept, and zero discriminant. Therefore, $\chi(\rho)$ has a unique zero at 
\begin{align*}
	\rho=\frac{(\alpha+\gamma\kappa+2\sqrt{\alpha\gamma\kappa})(-\alpha+\sqrt{\alpha\gamma\kappa})}{(\gamma\kappa-\alpha)(\alpha+\gamma\kappa+2\sqrt{\alpha\gamma\kappa})}=\rho^*.
\end{align*}
Hence, 
\begin{equation*}
    \rho<\rho^*\; \implies \; \chi(\rho)<0 \; \iff \;\beta_\Delta<\beta_c.
\end{equation*}
With \eqref{rel_betaBbetaD} this means that for $\rho<\rho^*$ it holds that $\beta_b<\beta_\Delta<\beta_c$. As the positive root $\beta_\Delta$ of $\Delta (\beta)$ lies on the right of the zero of $b(\beta)$, then for $\beta\geq\beta_\Delta$ we have $\Delta (\beta)\geq0$ and necessarily $b(\beta)<0$. As $c(\beta)>0$ (that is $\r0<1$) for $\beta<\beta_c$, the condition on $\beta$ to have two positive roots $I_\pm^*$ of~$\Phi(I^*)$ is $\beta_\Delta\leq\beta<\beta_c$. Dividing by $\beta_c$ and using the definition~\eqref{def:Rc} of $\mathcal{R}_C$, we obtain the condition on $\r0$ as in statement \textit{b)} of the theorem. The rest follows from Theorem~\ref{theo:bifdir}.
\end{proof}
\begin{figure}[t]
	\centering
	\includegraphics[width=0.9\textwidth]{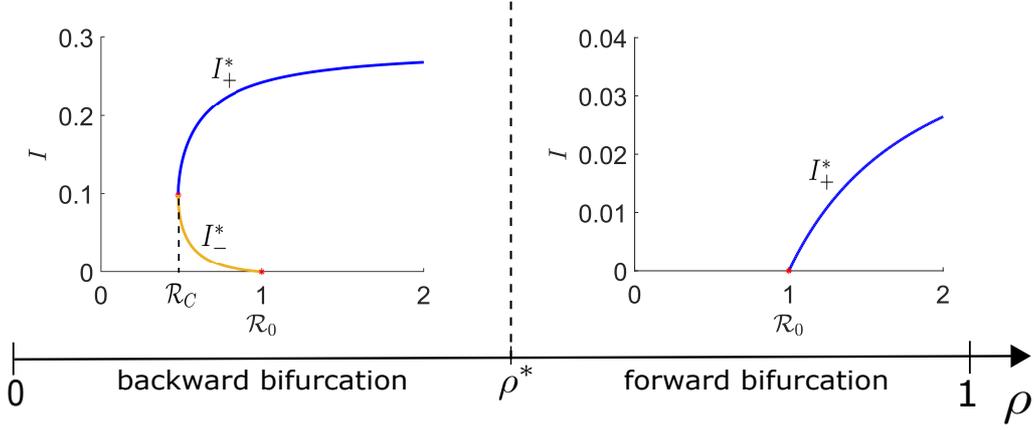}
	\caption{Visualization of statements in Theorem~\ref{Theorem1} and Theorem~\ref{Theorem2}: Endemic equilibria of system~\eqref{eq:reducedmodel} in dependence on $\rho$.}
	\label{fig:FigureX}
\end{figure}
The results of Theorem~\ref{Theorem1} and Theorem~\ref{Theorem2} are summarized in Fig.~\ref{fig:FigureX}. Local stability of the endemic equilibria can be determined evaluating the Jacobian matrix~\eqref{Jacobian} of the system at $\mathcal{E}_\pm$,
\[
J|_{\mathcal{E}_{\pm}}=
\left( {\begin{array}{ccc}
	-\beta I^*_\pm-\gamma & -\beta S_\pm^*-\gamma & \eta-\gamma\\
	\left(\rho-1\right)\beta I^*_\pm\,\, &\beta\left( \left(\rho-1\right)S_\pm^*+1-\left(1+\kappa\right)I^*_\pm\right)-\alpha\,\, & -\beta I^*_\pm \\
	0 & \alpha & -\eta \\	
	\end{array} } \right).
\]
The characteristic polynomial of $J|_{\mathcal{E}_\pm}$ is given by
\begin{equation}
\Pi(\lambda)=\lambda^3+\lambda^2a_2+\lambda a_1+a_0,
\label{charpol}
\end{equation}
where
\begin{align*}
a_2=&\beta I^*_\pm\left(2+\kappa\right)+\gamma+\alpha+2\beta S_\pm^*\left(1-\rho\right)-\beta+\eta\\
a_1=&2\beta^2I^*_\pm S_\pm^*\left(1-\rho\right)+\beta^2{I^*_\pm}^2\left(1+\kappa\right)+\gamma\alpha+\gamma\beta S_\pm^*\left(1-\rho\right)\\
&\;+{}\gamma\beta I^*_\pm\left(3-\rho\right)+\gamma\beta\frac{\alpha}{\eta}I^*_\pm+\eta\alpha+\eta\beta S^*\left(1-\rho\right)+3\eta\beta I^*_\pm\\
&\;+{}\alpha\beta I^*_\pm+\eta\gamma-\beta^2 I^*_\pm-\eta\beta-\gamma\beta\\
a_0=&\eta\beta^2 I^*_\pm\left(2\left(S_\pm^*\left(1-\rho\right)+I^*_\pm\right)-1\right)+\gamma\eta\beta(S_\pm^*(1-\rho)\\
&\;+{}I^*_\pm\left(3-\rho\right)-1)+(\eta-\gamma)\alpha\beta I^*_\pm\left(2-\rho\right)+\gamma\eta\alpha.
\end{align*}
With the Routh--Hurwitz criteria~\cite{murray} it follows that:
\begin{theorem}\label{eq:RH}
Let  $\mathcal{E}^*$ be an endemic equilibrium of system~\eqref{eq:reducedmodel} and let $J|_{\mathcal{E}^*}$ the coefficient matrix of the linearization of the system~\eqref{eq:reducedmodel} about $\mathcal{E}^*$. Then
$\mathcal{E}^*$ is LAS if the coefficients of the characteristic polynomial~\eqref{charpol} satisfy $a_2>0$, $a_0>0$, and  $a_1a_2>a_0$.
\end{theorem}
We enrich the analytical results obtained so far by means of  numerical investigations.

\section{Numerical bifurcation analysis}
From the analysis in Sect.~\ref{sec:results} it becomes clear that $\beta$ and $\rho$ are critical parameters affecting the qualitative behavior of the system. In this section we extensively investigate the $(\beta,\rho)$-parameter plane and numerically identify bifurcations of codimension 1 and 2. The remaining parameters are fixed $\eta=0.02$, $\alpha=0.2$, $\gamma=0.3$, such that $\rho^*\approx0.1976$. This choice is not motivated by a specific application or data, but highlights the rich dynamics the model can produce. All parametric portraits and simulations shown in what follows are produced using the numerical bifurcation software MATCONT~\cite{dhooge}. Background on bifurcation theory can be found e.g. in the seminal books by Yuri Kuznetsov~\cite{kuz} or Stephen Wiggins~\cite{wiggins}.\\
\ \\
At first we investigate regions of the $(\beta,\rho)$-plane where $\rho$ is constant and look for bifurcations in $\beta$ only. Figure~\ref{fig:Figure1} shows the $(\beta, I)$-plane, where curves of (the $I$-component of) endemic equilibria in dependence of $\beta$ are plotted for six different values of $\rho$, ordered by ascending $\rho$. Branches in blue (respectively, yellow) represent locally asymptotically stable (respectively, unstable) equilibria. 
Red asterisks mark either a transcritical bifurcation point (BP), a fold bifurcation point (LP), a supercritical Hopf bifurcation point ($\text{H}^{\text{sup}}$), a subcritical Hopf bifurcation point ($\text{H}^{\text{sub}}$) or 
a neutral saddle equilibrium (NS, no bifurcation point).
Figure~\ref{fig:Figure1} shows that as $\rho$ increases
there is a continuous change from a backward bifurcation (Fig.~\ref{fig:Figure1}(a--d)) to a forward bifurcation (Fig.~\ref{fig:Figure1}(e--f)), as the fold bifurcation point LP crosses the $\beta$-axis to the lower half-plane. This is in line with our analytical results from the previous section~(cf.~Theorems~\ref{theo:bifdir}--\ref{Theorem2}). Further, Hopf bifurcations occur on the equilibria branch corresponding to $\mathcal{E}_+$. We observe a change from a subcritical Hopf bifurcation (Fig.~\ref{fig:Figure1}(a)) to a supercritical Hopf bifurcation (Fig.~\ref{fig:Figure1}(b--e)).
The neutral saddle equilibrium (Fig.~\ref{fig:Figure1}(a--c)) gets closer to the fold bifurcation point, finally collides with it and turns into a second supercritical Hopf bifurcation point on the branch of $\mathcal{E}_+$ (Fig.~\ref{fig:Figure1}(d--e)). Increasing $\rho$, the two supercritical Hopf bifurcation points approach each other (Fig.~\ref{fig:Figure1}(d--e)), collapse into one point and eventually vanish (Fig.~\ref{fig:Figure1}(f)).\\

\begin{figure}[t]
	\centering
	\includegraphics[width=\textwidth]{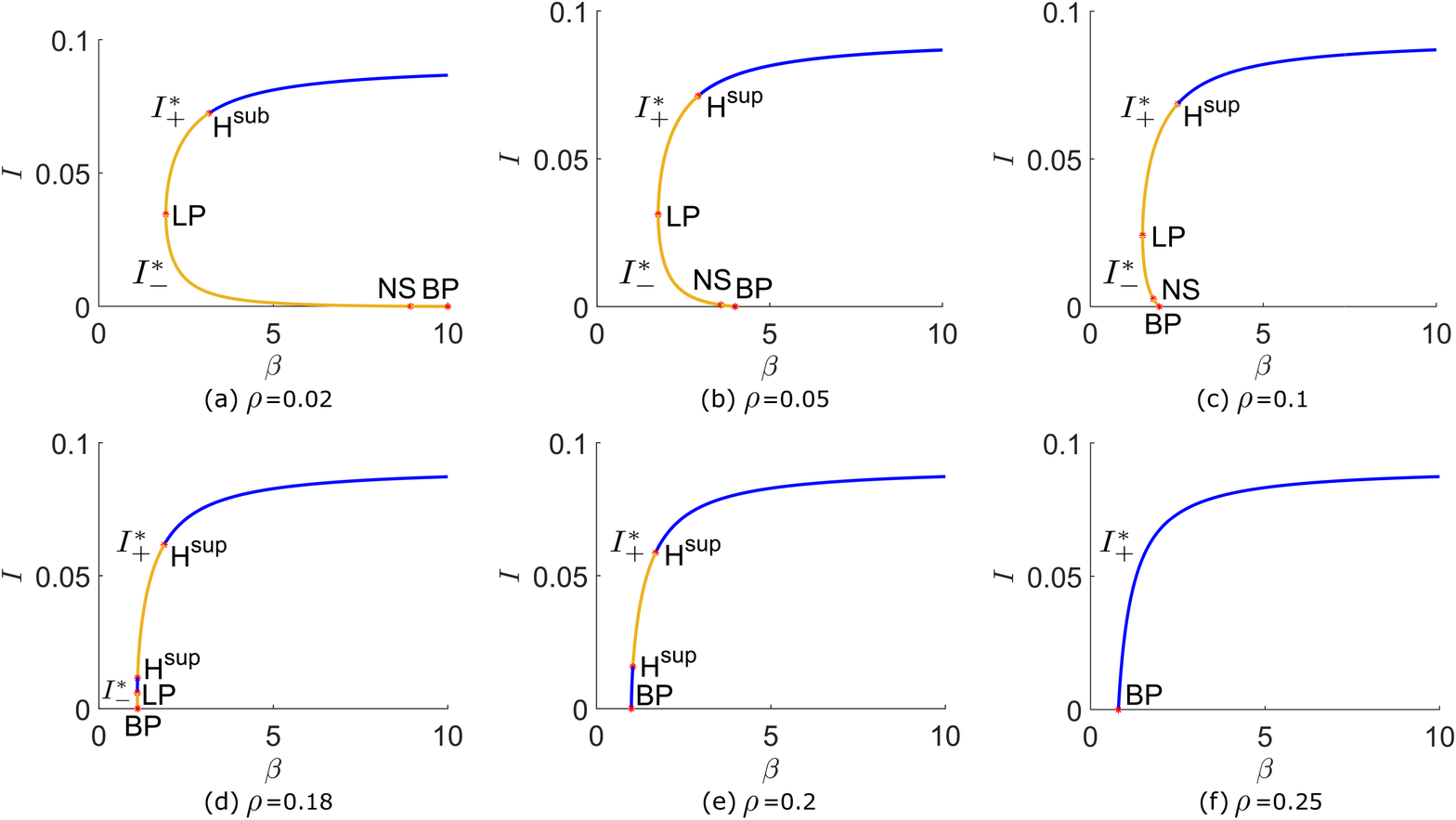}
	\caption{Curves of (the $I$-component of) endemic equilibria in dependence of $\beta$ are plotted for different values of $\rho$. Branches in blue (respectively, yellow) represent LAS (respectively, unstable) equilibria.}
	\label{fig:Figure1}	
\end{figure}

\noindent MATCONT provides additional information on the eigenvalues of the linearization about the endemic equilibria. In case of a backward bifurcation with a single Hopf bifurcation point (Fig.~\ref{fig:Figure1}(a--c)), both endemic equilibria are born unstable at the fold bifurcation. The unstable manifold of $\mathcal{E}_+$ is two-dimensional and the one of $\mathcal{E}_-$ is one-dimensional.  The two eigenvalues of the linearization about $\mathcal{E}_+$ with positive real part cross the imaginary axis at the Hopf bifurcation point so that $\mathcal{E}_+$ becomes LAS. In the situation where we observe two (supercritical, Fig.~\ref{fig:Figure1}(d--e)) Hopf bifurcation points on the branch of $\mathcal{E_+}$, this equilibrium was born LAS at the fold bifurcation. Two eigenvalues cross the imaginary axis to the right half-plane at the first Hopf bifurcation point and cross it again back to the left half-plane at the second Hopf bifurcation point.\\

\begin{figure}[t]
	\centering
	\includegraphics[width=\textwidth]{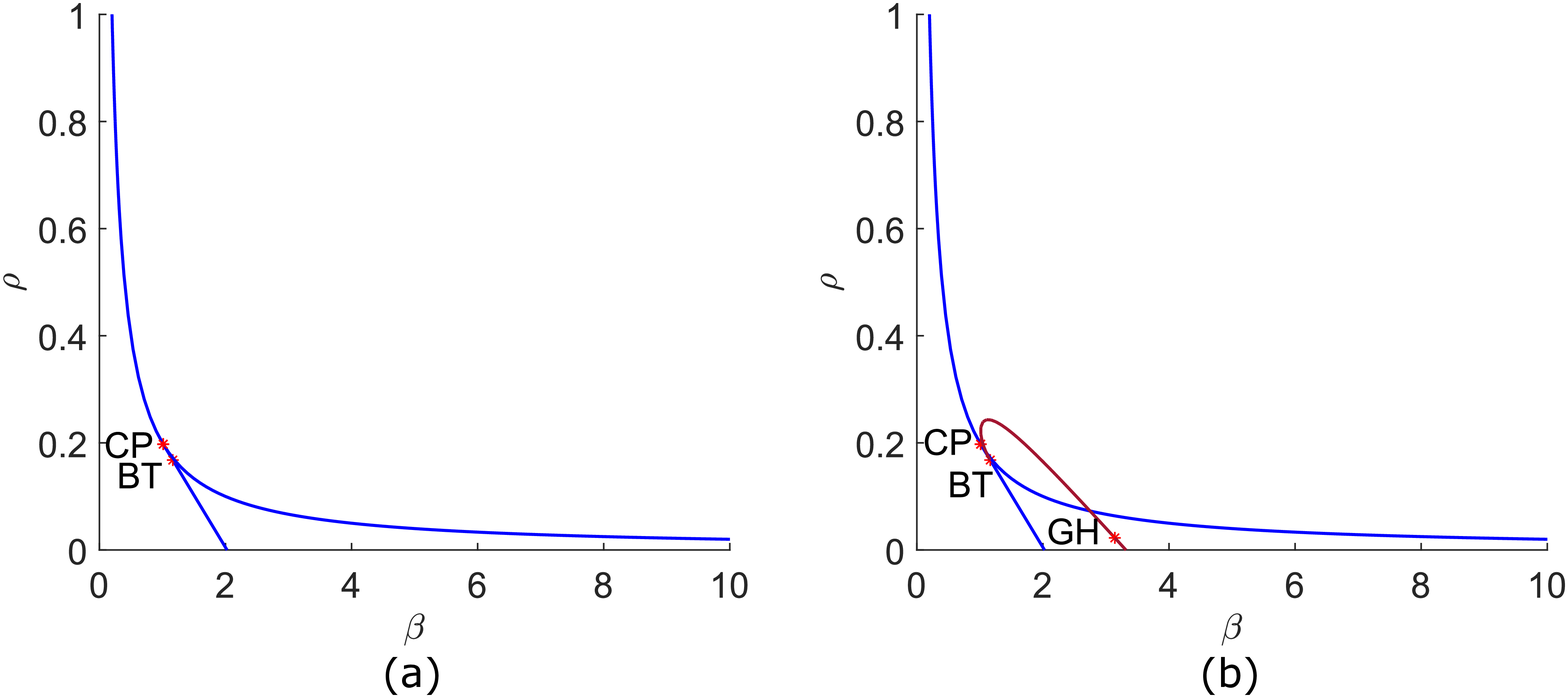}
	\caption{(a) Continuation of fold and transcritical bifurcation points with respect to $\beta$ and $\rho$; (b) Continuation of Hopf bifurcation points with respect to $\beta$ and $\rho$ starting from the Bogdanov--Takens point BT.}
	\label{fig:Figure2}
\end{figure}
\noindent The above observations help us to understand the parametric portrait in $\beta$ and $\rho$, and the arising codimension 2 bifurcations. Continuation of fold and transcritical bifurcation points with respect to $\beta$ and $\rho$ leads to Fig.~\ref{fig:Figure2}(a).
The curve of fold bifurcations (straight line) and the curve of transcritical bifurcations (curved line) meet at the point labeled CP. This point marks a cusp bifurcation where the normal form coefficient of the fold bifurcation vanishes. It divides the curve of transcritical bifurcations into the branch of forward bifurcations (upper part) and backward bifurcations (lower part). In addition, a Bogdanov--Takens (BT) bifurcation occurs on the curve of fold bifurcation points, where an additional eigenvalue approaches the imaginary axis. At the BT point the curve of Hopf bifurcation points meets the curve of fold bifurcations tangentially. Continuation of the Hopf bifurcation curve from BT leads to the red curve in Fig.~\ref{fig:Figure2}(b) (neutral saddle equilibria are not shown here). 
The Bogdanov--Takens point BT marks the point where the neutral saddle equilibrium turns into a supercritical Hopf bifurcation point or vice versa. The semi-elliptic shape of the Hopf curve leads to the occurrence of two Hopf bifurcations, in the continuation with respect to $\beta$, for values of $\rho$ above the Bogdanov--Takens point (cf. Fig.~\ref{fig:Figure1}(d--e)). Another codimension 2 bifurcation arises on the curve of Hopf bifurcations. A generalized Hopf bifurcation, where the first Lyapunov coefficient is zero and changes its sign, labeled GH marks the point where the Hopf bifurcation changes from subcritical (lower part) to supercritical (upper part). In what follows, we study the branches of limit cycles emerging from these Hopf bifurcation points. \\

\begin{figure}[t]
	\centering
	\includegraphics[width=\textwidth]{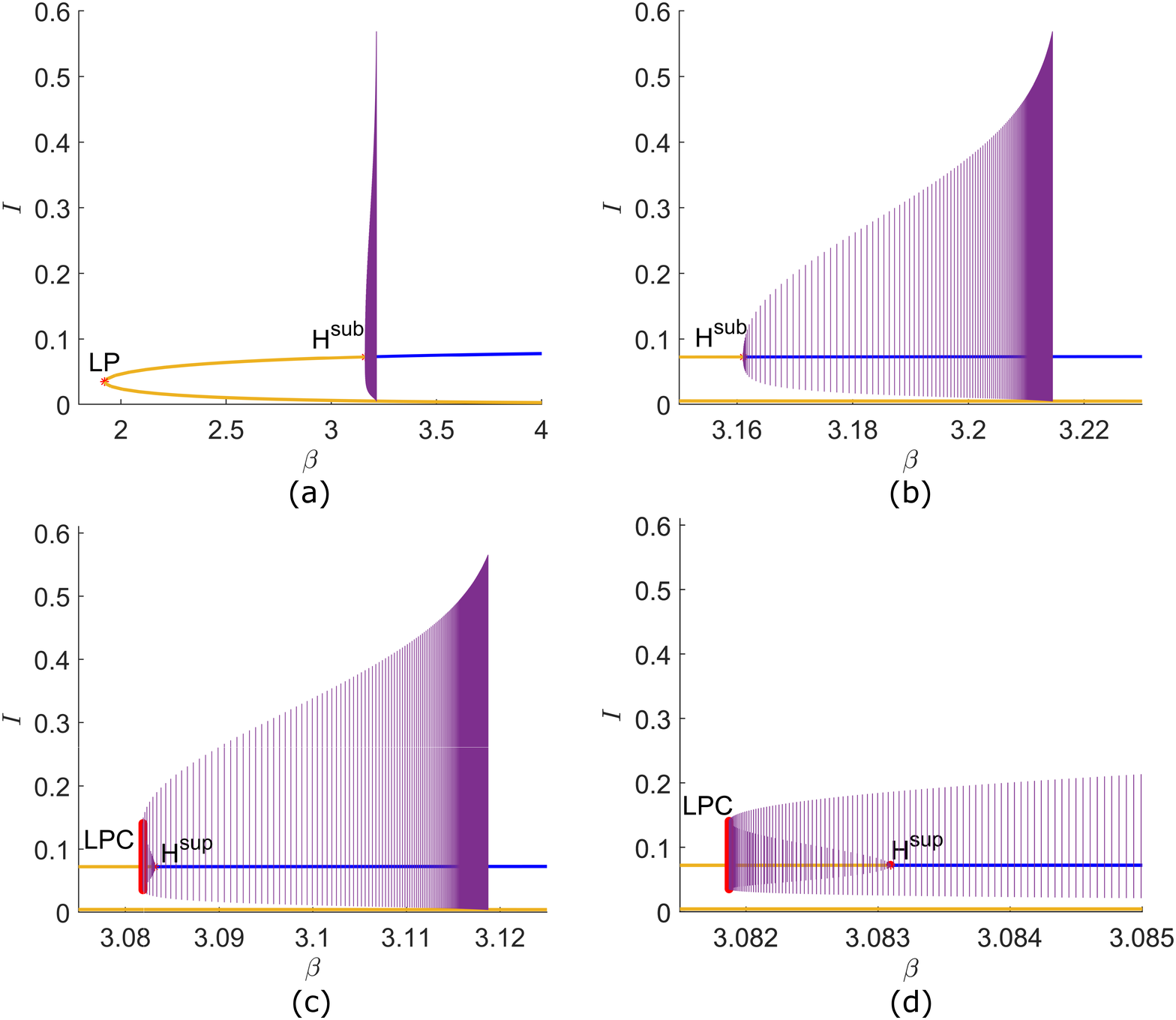}
	\caption{(a) Continuation of (the I-component of) limit cycles with respect to $\beta$ starting from subcritical Hopf bifurcation point for $\rho=0.02$; (b) Zoom of (a); (c) Continuation of (the I-component of) limit cycles with respect to $\beta$ starting from supercritical Hopf bifurcation point for $\rho=0.03$. The branch of stable limit cycles reaches a limit point cycle, folds back and becomes unstable; (d) Zoom of (c).}
	\label{fig:Figure5}
\end{figure} 
\noindent In Fig.~\ref{fig:Figure5}(a--b) we see that, for sufficiently small $\rho$, as $\beta$ increases, the amplitudes of the unstable limit cycles (with respect to the I-component) born at the subcritical Hopf bifurcation increase as well. The limit cycles get closer and closer to $\mathcal{E}_-$. This increases the period of the closed orbits, and a point moving on such orbits spends more and more time in the proximity of the equilibrium. The branch  of limit cycles disappears by colliding with the one-dimensional unstable manifold of $\mathcal{E_-}$ leading to a saddle homoclinic bifurcation~(more details are provided in the master thesis~\cite{JHmaster} of the first author).

Figure~\ref{fig:Figure5}(c--d) shows the same plot for a slightly larger value of $\rho$ such that the involved Hopf bifurcation is supercritical. The branch of stable limit cycles reaches a limit point cycle (LPC) then folds back and becomes unstable. Thus, there is a small interval of values of $\beta$ in which a stable limit cycle and an unstable one coexist. The unstable branch of limit cycles again vanishes through a saddle homoclinic bifurcation. Increasing $\rho$ further, the interval of values of $\beta$ where a stable limit cycle exists becomes larger, while the interval where an unstable limit cycle exists gets smaller and eventually vanishes (i.e. there is no LPC involved anymore). In this case it is the branch of stable limit cycles that vanishes through a saddle homoclinic bifurcation (not shown here, cf.~\cite{JHmaster}).

\begin{figure}[t]
	\centering
	\includegraphics[width=\textwidth]{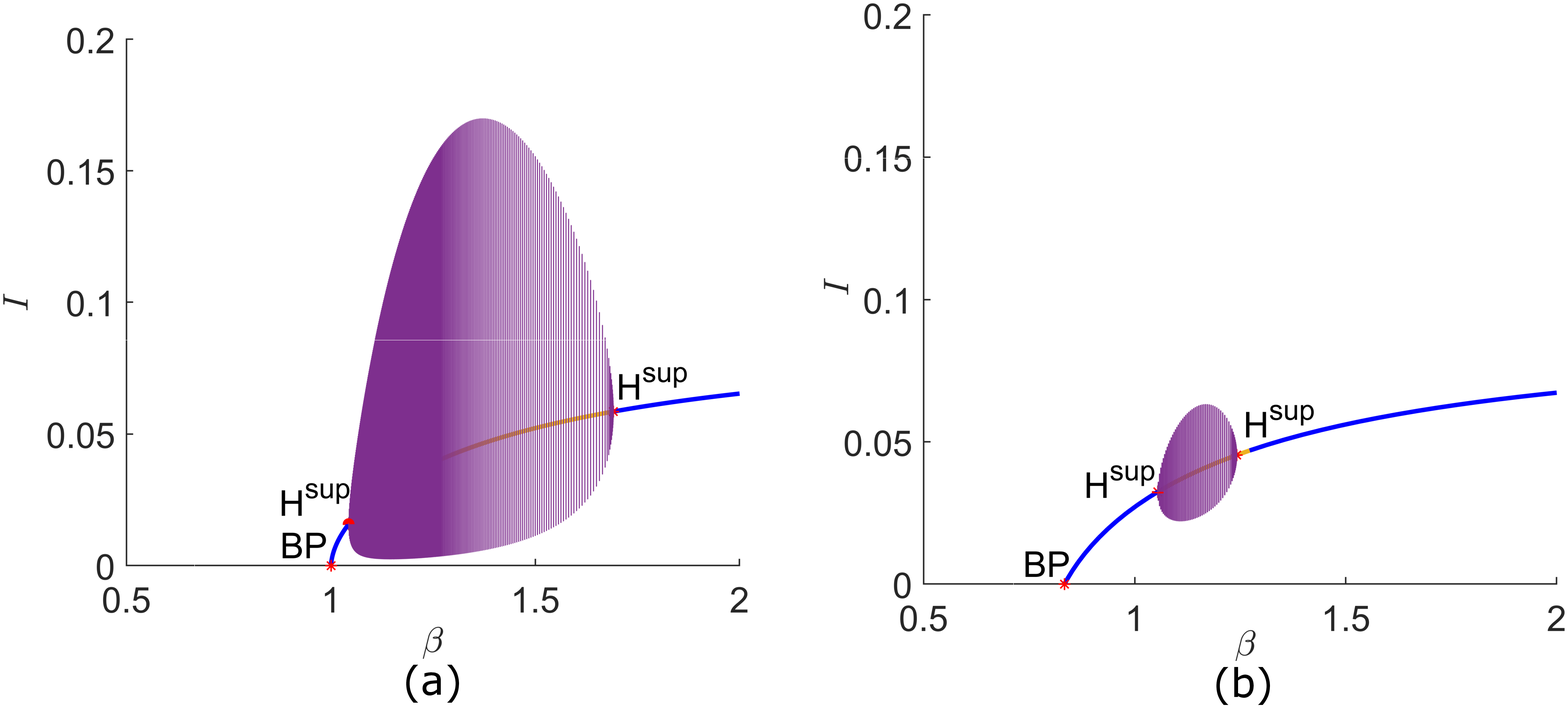}
	\caption{Continuation of (the I-component of) limit cycles with respect to $\beta$ starting from two supercritical Hopf bifurcation points for (a) $\rho=0.2$ and (b) $\rho=0.24$. The two branches of stable limit cycles meet and form an endemic bubble.}
	\label{fig:Figure9}
\end{figure}
Increasing $\rho$ further, two Hopf Bifurcation points, hence two branches of stable limit cycles exist, both vanishing through a saddle homoclinic bifurcation (not shown here, cf.~\cite{JHmaster}). For even larger values of $\rho$ there exists a bubble of limit cycles, starting and ending at the Hopf Bifurcation points (cf.~Fig.~\ref{fig:Figure9}(a)). Increasing $\rho$ the bubble becomes  smaller (cf.~Fig.~\ref{fig:Figure9}(b)), until the two Hopf points collide and the limit cycles disappear (cf.~Fig.~\ref{fig:Figure1}(f)).\\

\begin{figure}[t]
	\centering
	\includegraphics[width=\textwidth]{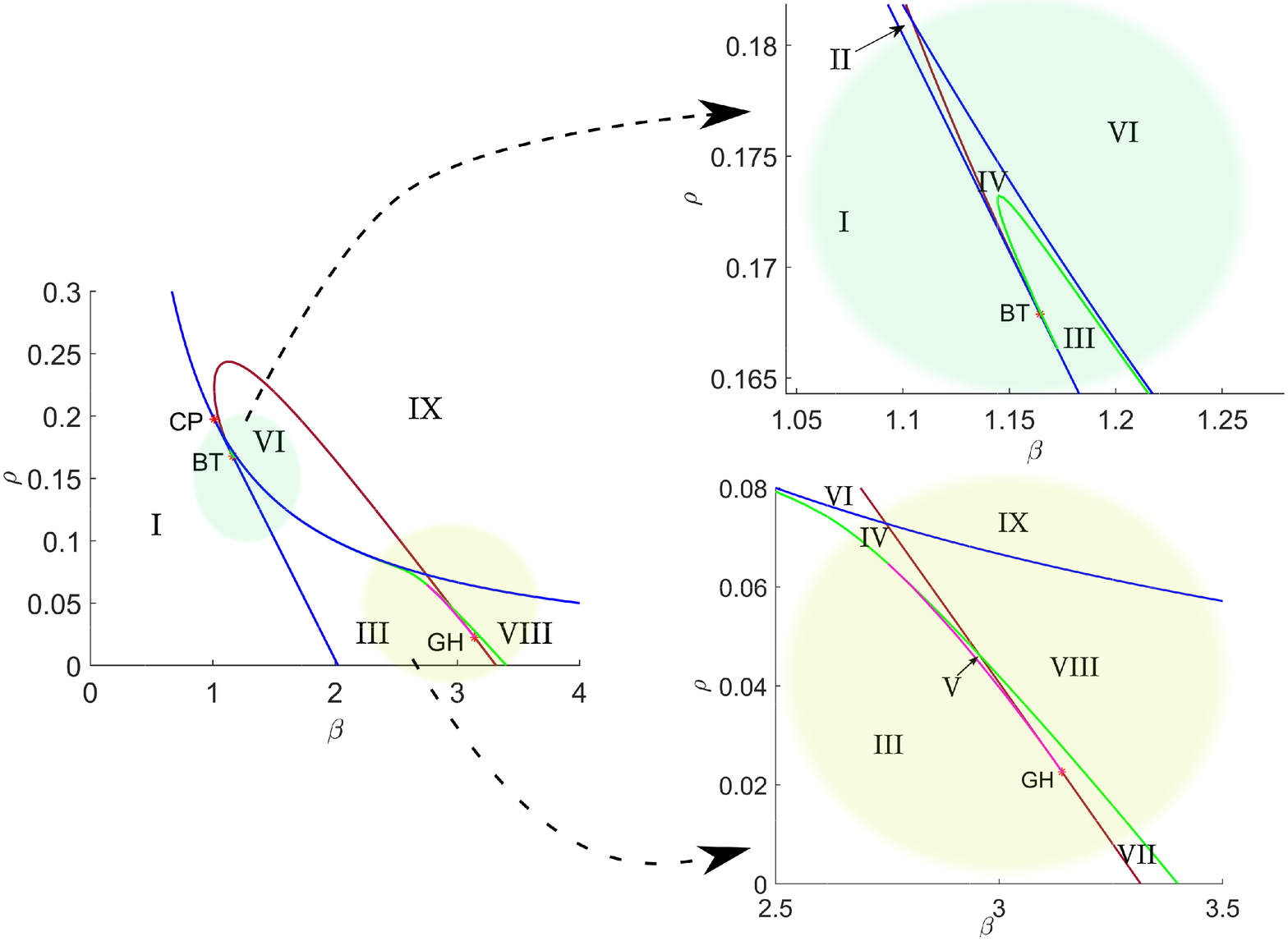}
	\caption{Parametric portrait with respect to $\beta$ and $\rho$. The different bifurcation curves subdivide the $(\beta,\rho)$-plane into regions. The green and yellow balloons show zoomed regions around the BT and GH point, respectively.}
	\label{fig:Figure12}
\end{figure}

\noindent Continuation of the curve of saddle homoclinic bifurcations (light green) and fold bifurcations of limit cycles (magenta) yield the complete parametric portrait in ~Fig.~\ref{fig:Figure12}. The different bifurcation curves subdivide the $(\beta,\rho)$-plane into regions where the system shows different qualitative behavior. We conclude this numerical study by a discussion of typical solution trajectories and orbits for each region.\\

\begin{figure}[b]
	\centering
	\includegraphics[width=\textwidth]{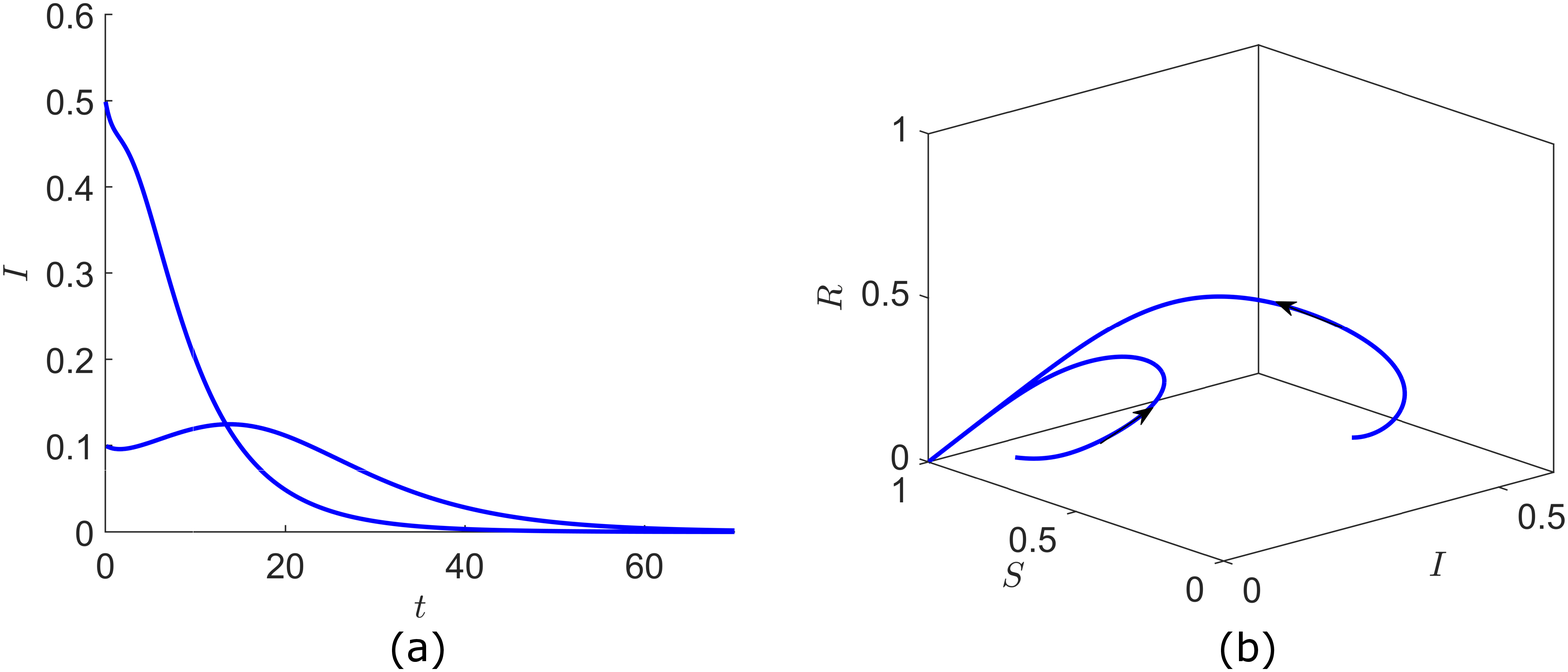}
	\caption{Dynamics in region I: Convergence to $\mathcal{E}_0$ shown in (a) $(t,I)$-plane and (b) $(S,I,R)$-phase space. For these simulations we set parameter values: $\beta=1,\, \rho=0.15,\, \alpha=0.2,\, \gamma=0.3$, and $\eta=0.02$; and initial conditions: $R_0=0,\, I_0=0.5$ and $I_0=0.1,\, S_0=1-I_0$.}
	\label{fig:Figure13}
\end{figure}

\noindent \textbf{Region I}: The disease-free equilibrium $\mathcal{E}_0$ is the only equilibrium and it is LAS (Fig.~\ref{fig:Figure13}).
Crossing the boundary to region IX leads to a transcritical bifurcation through which $\mathcal{E}_0$ loses stability and the LAS $\mathcal{E}_+$ arises. Crossing the boundary to region II results in a fold bifurcation, with the LAS equilibrium $\mathcal{E_+}$ and the unstable equilibrium $\mathcal{E}_-$ (with one-dimensional unstable manifold). Passing from region I to region III also leads to a fold bifurcation, through which both $\mathcal{E}_+$ (two-dimensional unstable manifold) and $\mathcal{E}_-$ (one-dimensional unstable manifold) arise unstable.\\

\begin{figure}[t]
	\centering
	\includegraphics[width=\textwidth]{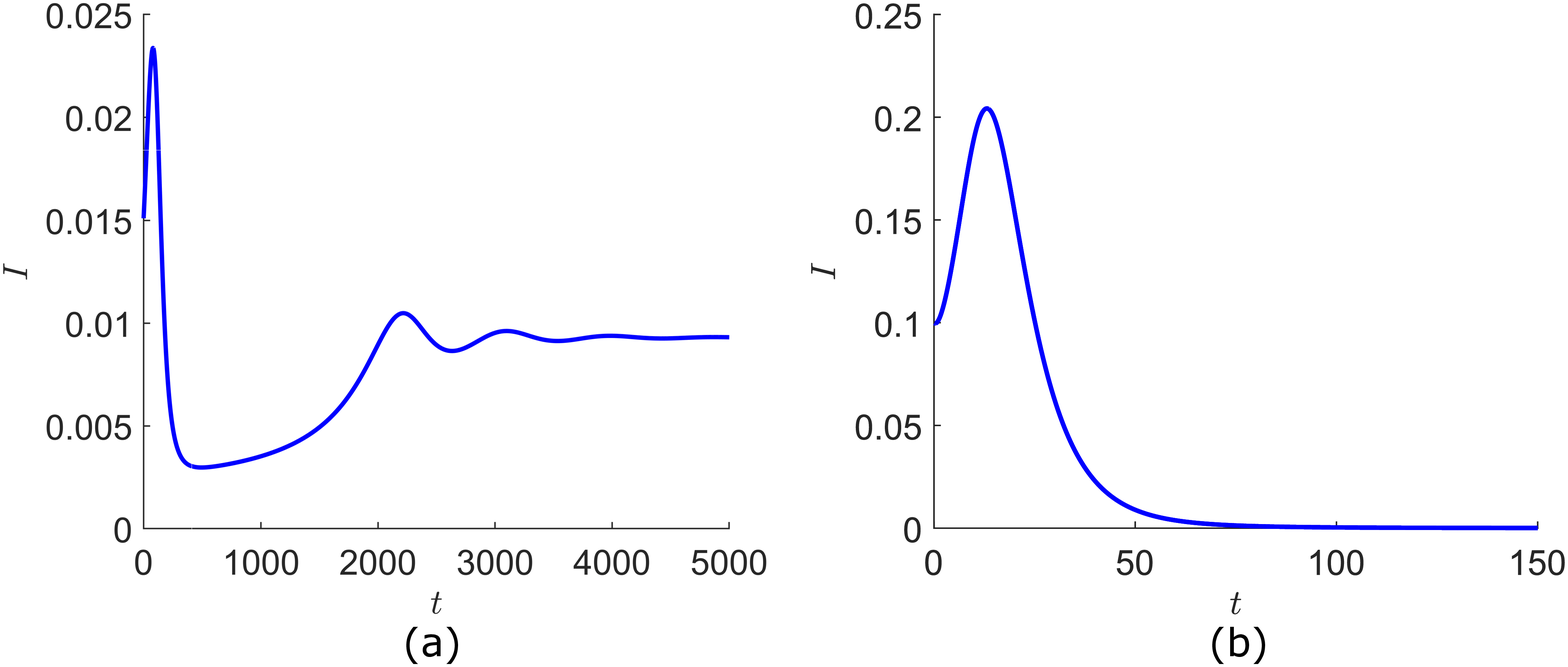}
	\caption{Dynamics in region II: Bistability of $\mathcal{E}_0$ and $\mathcal{E}_+$. For these simulations we set parameter values: $\beta=1.1,\, \rho=0.181,\, \alpha=0.2, \,\gamma=0.3$, and $\eta=0.02$, and initial conditions: (a) $R_0=0.1,\, S_0=0.85,\, I_0=0.015$; (b) $R_0=0,\, S_0=0.9,\, I_0=0.1$.}
	\label{fig:Figure14}
\end{figure}

\noindent \textbf{Region II}: The disease-free equilibrium $\mathcal{E}_0$ is LAS. Both endemic equilibria $\mathcal{E}_\pm$ exist, $\mathcal{E}_+$ being LAS and $\mathcal{E}_-$ being unstable with one-dimensional unstable manifold. Figure~\ref{fig:Figure14} shows the bistability of $\mathcal{E}_0$ and $\mathcal{E}_+$. Crossing the boundary to region IX leads to a transcritical backward bifurcation through which $\mathcal{E}_0$ loses stability and $\mathcal{E}_-$ moves out of the first quadrant. Thus, we observe a discontinous jump of orbits from the DFE $\mathcal{E}_0$ to the already quite large endemic equilibrium $\mathcal{E}_+$, which is known as hysteresis. Passing from region II to region IV leads to a Hopf bifurcation through which $\mathcal{E}_+$ becomes unstable with two-dimensional unstable manifold, and a stable periodic orbit emerges. Crossing the boundary between region II and region VI (intersection of Hopf and transcritical bifurcation curve) the above bifurcations happen simultaneously.\\

\begin{figure}[t]
	\centering
	\includegraphics[width=\textwidth]{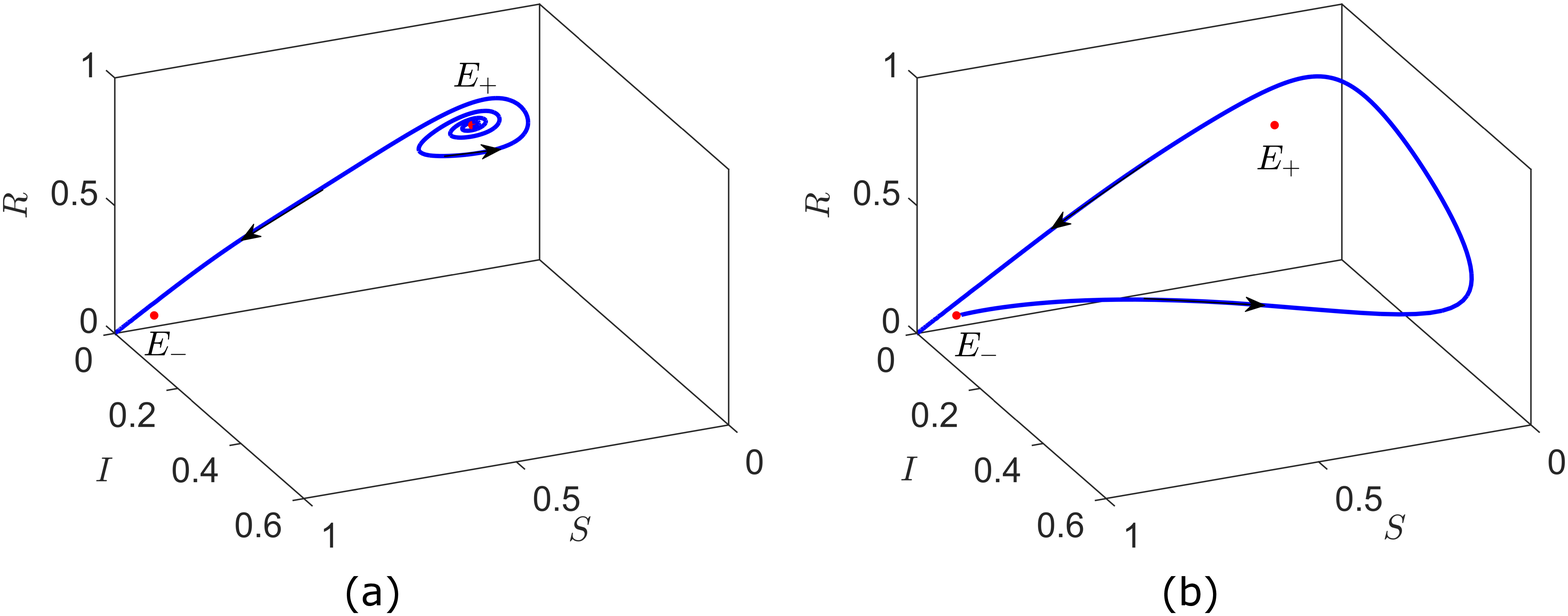}
	\caption{Dynamics in region III:  Convergence to the DFE $\mathcal{E}_0$. Both endemic equilibria $\mathcal{E}_\pm$ are unstable. For these simulations we set parameter values: $\beta=2.5,\, \rho=0.05,\, \alpha=0.2,\, \gamma=0.3$, and  $\eta=0.02$; and initial conditions: (a) $R_0=0.65,\, S_0=0.207,\, I_0=0.065$; (b) $R_0=0.05,\, S_0=0.9,\, I_0=0.006$.}
	\label{fig:Figure15}
\end{figure}

\noindent \textbf{Region III}: The system has three equilibria (cf.~Fig.~\ref{fig:Figure15}). The DFE $\mathcal{E}_0$ is LAS and both endemic equilibria $\mathcal{E}_\pm$ are unstable, the unstable manifold of $\mathcal{E}_+$ being two-dimensional, that of $\mathcal{E}_-$ one-dimensional. We observe the phenomenon of excitability: Solutions starting close to $\mathcal{E}_-$ follow its unstable manifold, which leads to a large epidemic outbreak before converging to $\mathcal{E}_0$ (cf.~Fig.~\ref{fig:Figure15}(b)). Crossing the boundary to region IV results into a saddle homoclinic bifurcation, which produces a stable periodic orbit. Passing the boundary to region V, a fold bifurcation of limit cycles takes place and a stable and an unstable periodic orbit are born. Moving from region III to region VII leads to a subcritical Hopf bifurcation, $\mathcal{E}_+$ becomes LAS and an unstable limit cycle arises.\\

\begin{figure}[t]
	\centering
	\includegraphics[width=\textwidth]{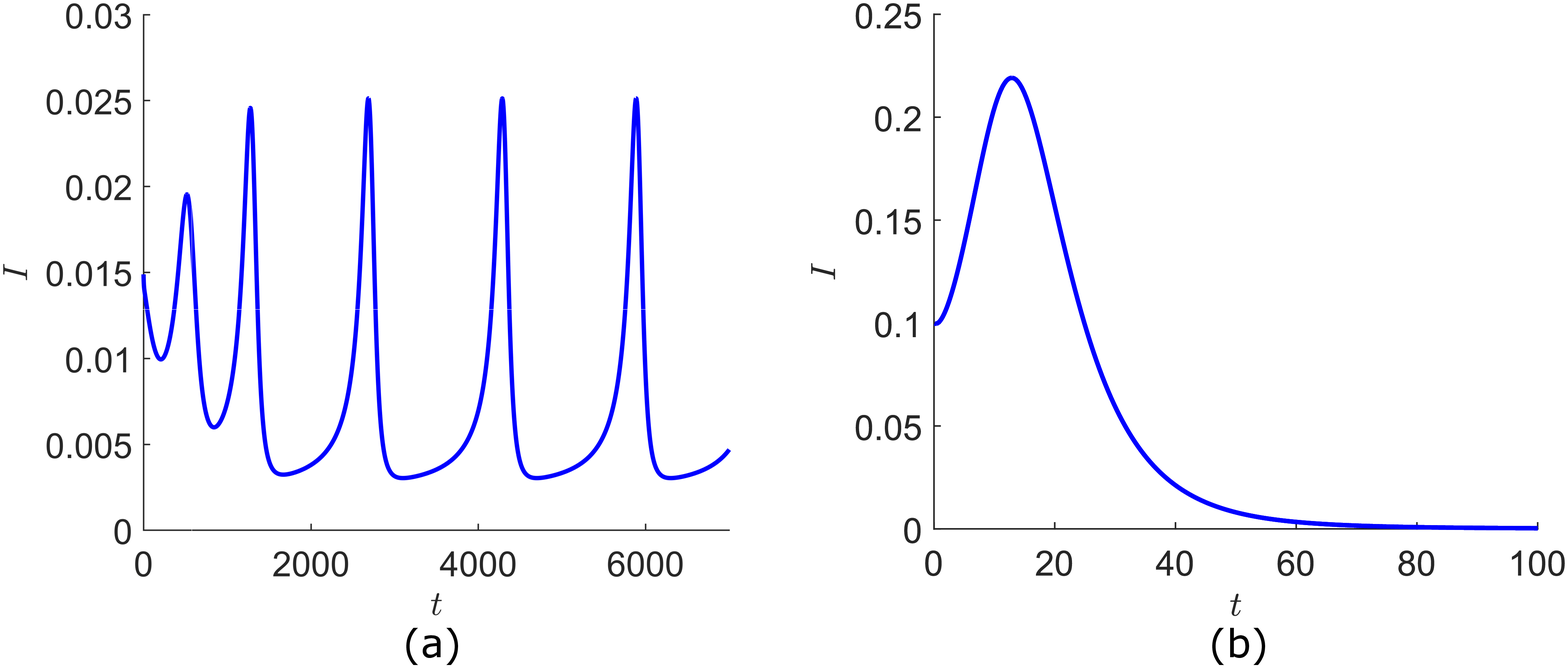}
	\caption{Dynamics in region IV: Bistability of $\mathcal{E}_0$ and a limit cycle. For these simulations we set parameter values: $\beta=1.14,\, \rho=0.174,\, \alpha=0.2,\, \gamma=0.3$ and $\eta=0.02$; and initial conditions: (a) $R_0=0.15,\, S_0=0.82,\, I_0=0.015$; (b) $R_0=0,\, S_0=0.9,\, I_0=0.1$.}
	\label{fig:Figure16}
\end{figure}

\noindent \textbf{Region IV}: The disease-free equilibrium $\mathcal{E}_0$ is LAS and both endemic equilibria $\mathcal{E}_\pm$ are unstable, 
the unstable manifold of $\mathcal{E}_+$ being two-dimensional, that of $\mathcal{E}_-$ one-dimensional. Additionally, a stable limit cycle exists. This leads to a bistability of the disease-free equilibrium and a periodic solution (cf.~Fig.~\ref{fig:Figure16}). 
Crossing the boundary to region V results into a saddle homoclinic bifurcation, which produces a second (unstable) periodic orbit. Passing the boundary to region VI, a transcritical bifurcation occurs, through which $\mathcal{E}_0$ loses stability and $\mathcal{E}_-$ moves out of the first quadrant (hysteresis, cf. Region II). Moving from region IV to region VIII leads to a Hopf bifurcation, $\mathcal{E}_+$ becomes LAS and the stable periodic orbit vanishes. Crossing the boundary to region IX (intersection of Hopf and transcritical bifurcation curve) the last two bifurcations happen simultaneously.\\

\begin{figure}[t]
	\centering
	\includegraphics[width=\textwidth]{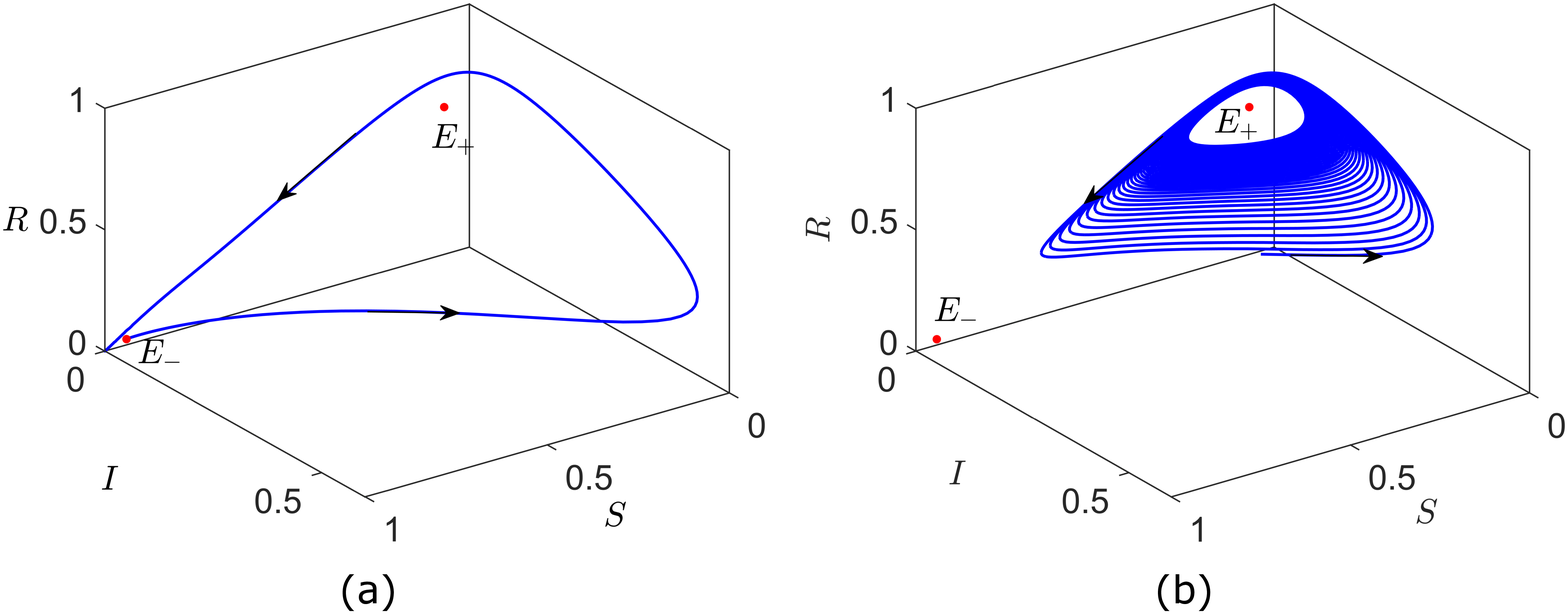}
	\caption{Dynamics in region V: Bistability of $\mathcal{E}_0$ and a limit cycle. For these simulations we set parameter values: $\beta=2.952,\, \rho=0.046,\, \alpha=0.2,\, \gamma=0.3,\,\eta=0.02$ and initial conditions: (a) $R_0=0.03,\, S_0=0.94,\, I_0=0.003$; (b) $R_0=0.3,\, S_0=0.3,\, I_0=0.2$.}
	\label{fig:Figure17}
\end{figure}
\noindent \textbf{Region V}: The dynamical behavior in this region is similar to that in Region IV (cf.~Fig.~\ref{fig:Figure17}). 
Additionally to the stable limit cycle, an unstable limit cycle exists, separating the basin of attraction of $\mathcal{E}_0$ and of the stable limit cycle. Like in Region III, we observe the phenomenon of excitability for orbits starting close to $\mathcal{E}_-$ (cf.~Fig.~\ref{fig:Figure17}(a)).  Moving to region VII leads to a supercritical Hopf bifurcation. The stable limit cycle disappears and $\mathcal{E}_+$ becomes LAS. Crossing the boundary to region VIII (intersection of Hopf and saddle homoclinic bifurcation curve) the supercritical Hopf bifurcation happens simultaneously with a saddle homoclinic bifurcation through which also the unstable limit cycle vanishes.\\

\begin{figure}[t]
	\centering
	\includegraphics[width=\textwidth]{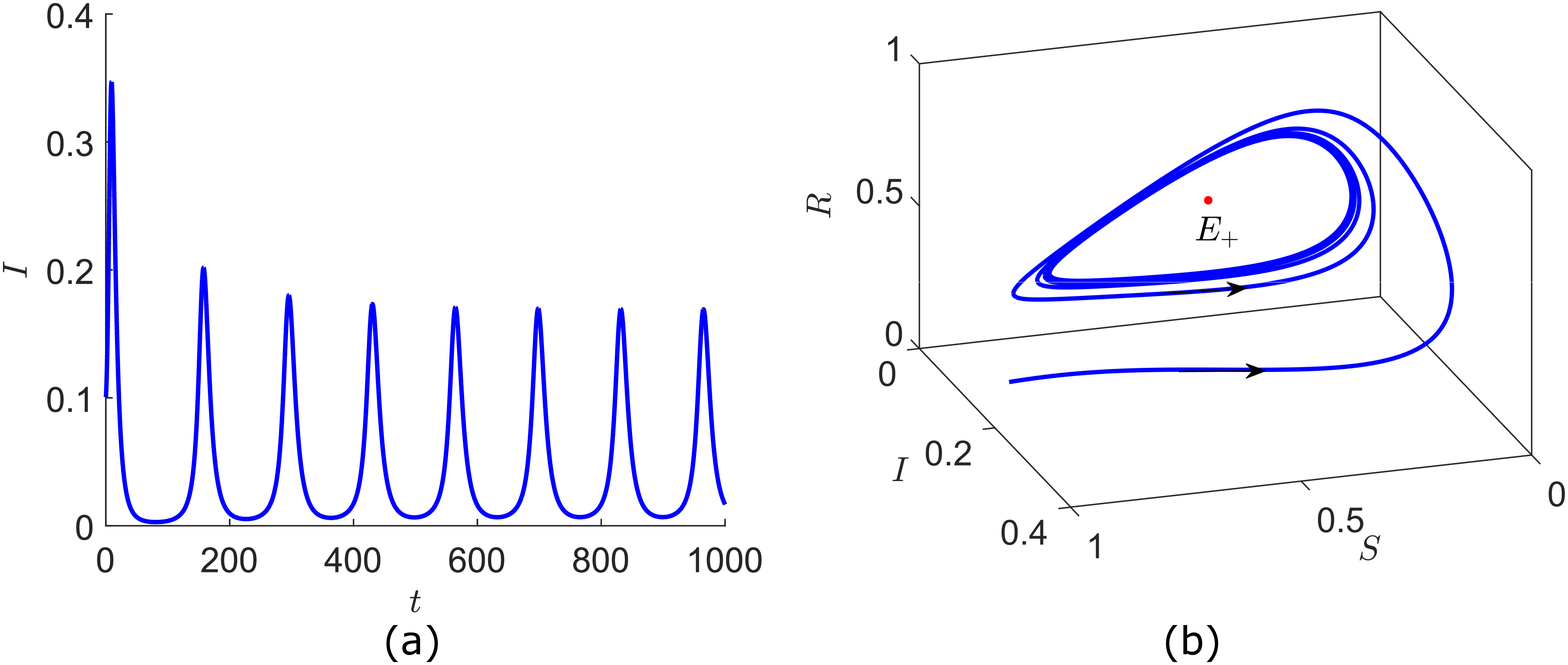}
	\caption{Dynamics in region VI: Convergence to a stable limit cycle shown in (a) $(t,I)$-plane and (b) $(S,I,R)$-phase space. For these simulations we set parameter values: $\beta=1.4,\, \rho=0.2,\, \alpha=0.2,\, \gamma=0.3,\, \eta=0.02$ and initial conditions: $R_0=0,\, S_0=0.9,\, I_0=0.1$.}
	\label{fig:Figure18}
\end{figure} 
\noindent \textbf{Region VI}: The disease-free equilibrium $\mathcal{E}_0$ is unstable, and so is also the only endemic equilibrium $\mathcal{E}_+$. There exists a stable limit cycle
to which solutions converge (cf.~Fig.~\ref{fig:Figure18}).
Moving from region VI to region IX leads to a supercritical Hopf bifurcation, the stable limit cycle vanishes and $\mathcal{E}_+$ becomes LAS.\\

\begin{figure}[t]
	\centering
	\includegraphics[width=\textwidth]{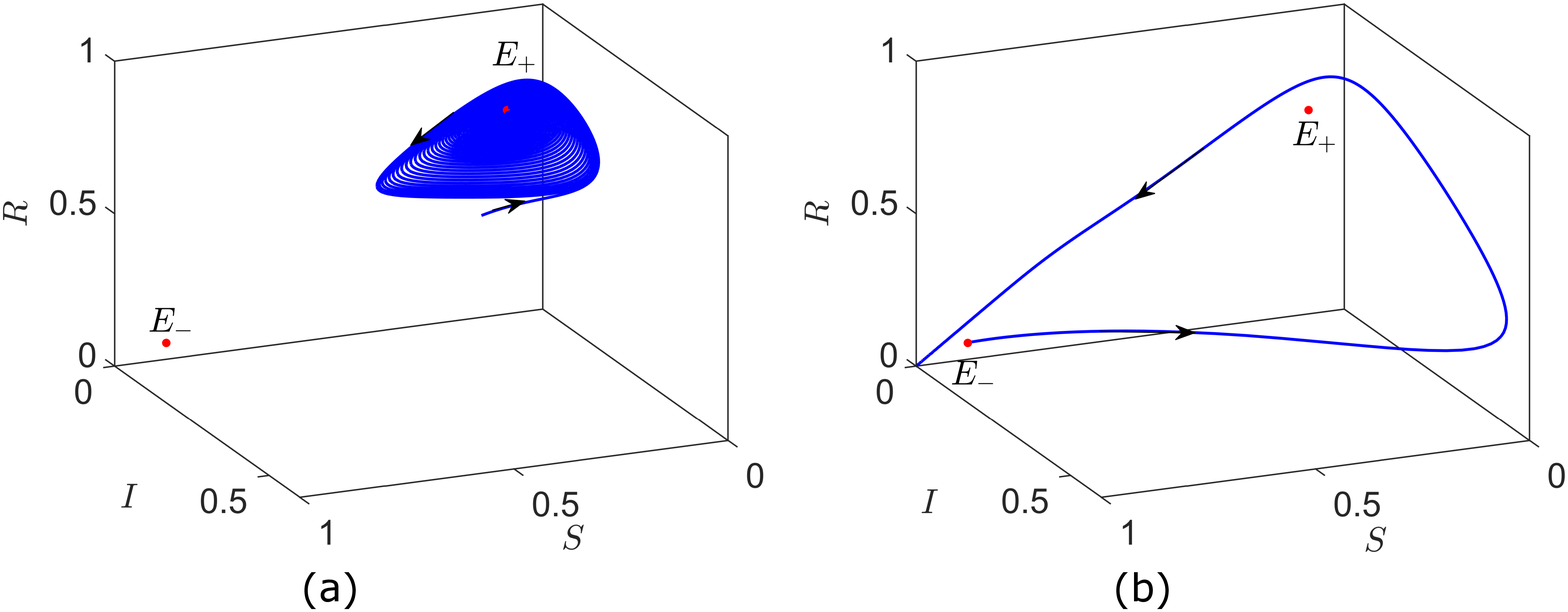}
	\caption{Dynamics in region VII: Bistability of $\mathcal{E}_0$ and $\mathcal{E}_+$. For these simulations we set parameter values: $\beta=3.27,\, \rho=0.01,\, \alpha=0.2,\, \gamma=0.3,\, \eta=0.02$ and initial conditions: (a) $R_0=0.5,\, S_0=0.3,\, I_0=0.2$; (b) $R_0=0.06,\, S_0=0.88,\, I_0=0.006$.}
	\label{fig:Figure19}
\end{figure}

\noindent \textbf{Region VII}: The equilibria $\mathcal{E}_0$ and $\mathcal{E}_+$ are both LAS (cf.~Fig.~\ref{fig:Figure19}),  whereas the third equilibrium $\mathcal{E}_-$ is unstable with one-dimensional unstable manifold. Furthermore, there exists an unstable periodic orbit that separates the basin of attraction of $\mathcal{E}_0$ and $\mathcal{E}_+$. Like in Region III, we observe the phenomenon of excitability for orbits starting close to $\mathcal{E}_-$ (cf.~Fig.~\ref{fig:Figure19}(b)). Crossing the boundary to region VIII leads to a saddle homoclinic bifurcation and the unstable periodic orbit vanishes.\\

\noindent \textbf{Region VIII}: In this region, we have bistability of  $\mathcal{E}_0$ and $\mathcal{E}_+$~(cf.~Fig.~ \ref{fig:Figure21}(a)). The second endemic equilibrium $\mathcal{E}_-$ is unstable with one-dimensional unstable manifold connecting to $\mathcal{E}_+$. Crossing the boundary to region IX, a transcritical bifurcation takes place, so that $\mathcal{E}_0$ loses stability and $\mathcal{E}_-$ leaves the positive quadrant (hysteresis, cf. Region II).\\

\begin{figure}[t]
	\centering
	\includegraphics[width=\textwidth]{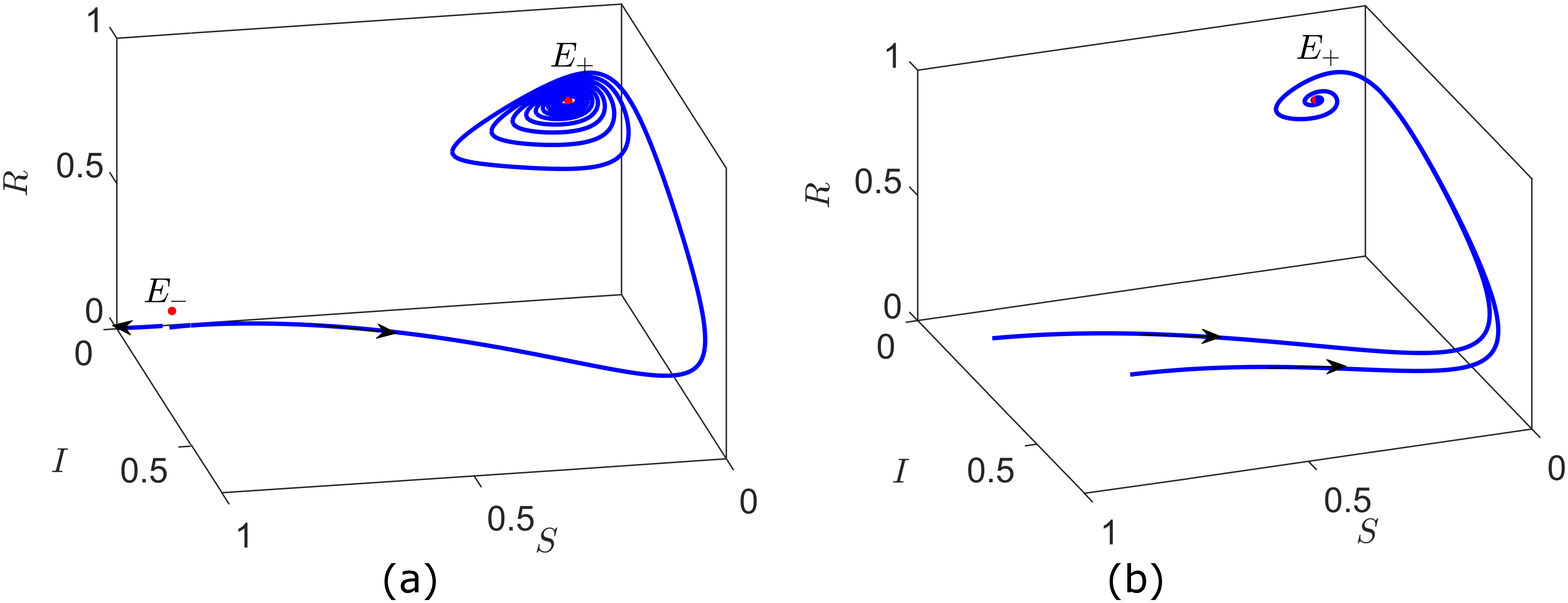}
	\caption{(a) Dynamics in region VIII: Bistability of $\mathcal{E}_0$ and $\mathcal{E}_+$. 
	For these simulations we set parameter values: $\beta=3.4, \rho=0.01,\, \alpha=0.2,\, \gamma=0.3,\, \eta=0.02$ and initial values: $R_0=0,\, S_0=0.9,\, I_0=0.001$ and $I_0=0.01$. (b) Dynamics in region IX: Convergence to the unique endemic equilibrium $\mathcal{E}_+$. For these simulations we set parameter values: $\beta=2.5,\, \rho=0.25,\, \alpha=0.2,\, \gamma=0.3,\, \eta=0.02$ and initial conditions: $R_0=0,\, I_0=0.1$ and $I_0=0.3,\, S_0=1-I_0$.}
	\label{fig:Figure21}
\end{figure}
\noindent\textbf{Region IX}: The DFE $\mathcal{E}_0$ is unstable, and the unique endemic equilibrium $\mathcal{E}_+$ is LAS~(cf.~Fig.~\ref{fig:Figure21}(b)).\\
\ \\
\noindent So far we have investigated the $(\beta,\rho)$-parameter plane fixing the values of $\alpha,\,\gamma,\,\eta$. One might ask how the qualitative behavior of system~\eqref{eq:reducedmodel} is affected by the particular choice of these three parameters. For instance, we constructed the parametric portrait from Fig.~\ref{fig:Figure12} for different values of $\eta$ (same was done for $\alpha$ and $\gamma$, though not shown here). The result is shown in Fig.~\ref{fig:Figure22}. Reducing $\eta$ (Fig.~\ref{fig:Figure22}(a)) region VI gets significantly larger. Slowing down the transition from $R$ to $S$ makes the model closer to a SIRS system with delay (cf.~\cite{Taylor,MVBsiap}) enhancing the occurrence of oscillations. Moreover, the generalized Hopf point in Fig.~\ref{fig:Figure22}(b) moves to the fourth quadrant. In contrast, increasing $\eta$ (Fig.~\ref{fig:Figure22}(c)) region VI gets significantly smaller, whereas the points GH, CP, BT move up along the bifurcation curves.

\begin{figure}[t]
	\centering
	\includegraphics[width=\textwidth]{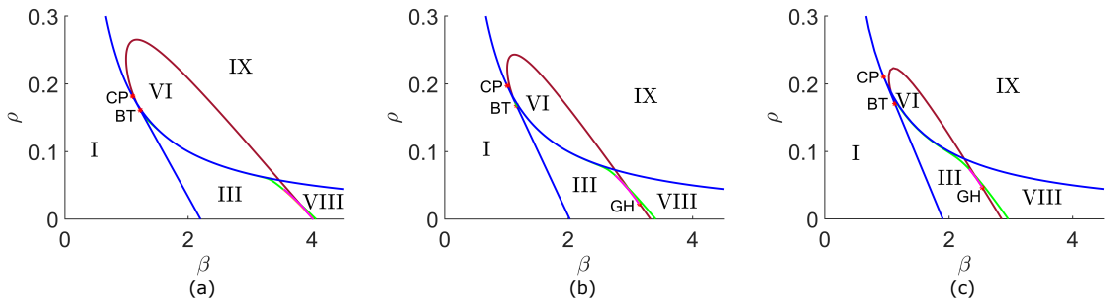}
	\caption{Parametric portrait in $\beta$ and $\rho$, for $\alpha=0.2$, $\gamma=0.3$, and (a) $\eta=0.016$; (b)  $\eta=0.02$; (c)=0.024.}
	\label{fig:Figure22}
\end{figure}

\clearpage
\section{Discussion}
In this work we have presented a mathematical model for transmission dynamics with two nonlinear stages of contagion, applying analytical and numerical methods to analyze its qualitative behavior. The two-stage contagion is combined with a renewal of the susceptible compartment, modeled by (i) the waning of immunity in inactive/resistant individuals and (ii) the fading of infection in those individuals who are weakened after the first contact with the promoting/infected community. This leads to rich dynamics, including bistability of equilibria or bistability of an equilibrium and a periodic solution, discontinuous regime shifts through hysteresis effects, and excitability. Thus, the multi-stage nature of social contagion processes might explain some of the complex phenomena observed in social dynamics (see e.g. the irregular~\cite{ITApool} or periodic~\cite{USpool} outcomes in political elections, or the emergence of new trends in the usage of social media~\cite{RiseSocialMedia}). In a previous study by Guy Katriel~\cite{kat} similar properties were determined for a two-stage contagion model with demographic turnover. This lets us conjecture that the rich dynamics observed in our work and in~\cite{kat} is due to the coupling of a two-stage contagion process with any (demographic or "immunological") source of renewal of the susceptible population. Despite of the analogies with Katriel's work, the combination of waning of immunity and fading of infection in our model leads to additional analytical complexity. The analytical advantage of Katriel's model was possibly due to the choice of the birth rate in the $S$-compartment matching with the death rates of all compartments. Katriel's model shows also both backward and forward bifurcation, and in the latter case it behaves like a one-stage contagion model for any choice of $\beta$. In contrast, our model shows Hopf bifurcations, hence periodic solutions, also in case of a forward bifurcation (cf.~Fig. \ref{fig:Figure12}(e)). Single phenomena which can be observed in our model have been previously found also for variations of the classical one-stage contagion models. For example SIRS models with delayed loss of immunity~\cite{Taylor} naturally show stable periodic solutions, whereas bistability of equilibria was found e.~g.~in models with exogenous reinfection~\cite{chavez2004} or imperfect vaccination~\cite{brauer2004}.\\
\ \\
In this work we have focused on the investigation of the qualitative properties of a simple two-stage contagion model. 
We have numerically investigated the parameter space, focusing in particular on the effects of the transmission rate ($\beta$) and the probability of a perfect contact $(\rho)$. Of course the study could be repeated deriving the parametric portrait of the system~\eqref{eq:reducedmodel} with respect to the other model parameters as well. Moreover, we see three possible generalizations of our model: (i) the assumption that $\beta=\beta_2$ could be relaxed, and e.g. $W$-individuals might be assumed to have a higher susceptibility than $S$-individuals; (ii) individuals in the $W$-compartment could also be contagious (cf.~also~\cite{cris}); (iii) the model could also include births and deaths, in addition to waning/fading processes. All these variations would make the analytical investigations more challenging; however, a similar numerical investigation as presented in this work could be performed.
Thinking of applications and comparison with data, the major  limitation of our work is due to the deterministic approach. Dividing a population into a few homogeneous compartments, without taking into account interpersonal variability, is indeed a major simplification of reality. Refining our approach, agent-based modeling~\cite{deff} and complex networks~\cite{albi} could be used. In certain cases, previous works based on these methods also included two-stage contagion~\cite{hase,melnik}.

\section*{Acknowledgement}
The authors are supported by the LOEWE focus CMMS.
\begin{appendices}
\section{}
In the proof of Theorem~\ref{theo:bifdir} we referred to the  following result by Castillo-Chavez and Song~\cite{chavez2004} proved using center-manifold theory.

\begin{theorem}\label{theo:chavez}
    Let $f\in C^2(\mathbb{R}^n\times \mathbb{R}, \mathbb{R}^n)$. Consider the following general system of ODEs with a parameter $\beta$
	\begin{align}
	\frac{dx}{dt}=f(x,\beta)\label{eq:appendix_system}.
	\end{align}
	Without loss of generality assume that $x_0=0$ is an equilibrium point of the system, that is, $f(0,\beta)=0$ for all $\beta \in \mathbb{R}$. Assume the following:
	\begin{itemize}
		\item The linearization of the system \eqref{eq:appendix_system} $\mathcal{A}:=D_xf(0,0)=\left (\frac{\partial f_i}{\partial x_j}(0,0)\right )$ has zero as a simple eigenvalue and all other eigenvalues of $\mathcal{A}$ have negative real parts.
		\item The matrix $\mathcal{A}$ has a non-negative right eigenvector $w$ and a left eigenvector $v$ each corresponding to the zero eigenvalue.
	\end{itemize}
	Let $f_k$ be the $k$-th component of $f$ and
	\begin{align}
	a&=\sum_{k,i,j=1}^{n}v_kw_iw_j\frac{\partial^2 f_k}{\partial x_i\partial x_j}(0,0)\notag\\
	b&=\sum_{k,i=1}^n v_kw_i\frac{\partial^2 f_k}{\partial x_i\partial \beta}(0,0)\notag.
	\end{align}
	Then, the local dynamics of the system around $0$ is completely determined by the signs of $a$ and $b$:
	\begin{enumerate}
		\item $a>0,\,b>0$. When $\beta<0$ with $|\beta|\ll 1$, $0$ is locally asymptotically stable and there exists a positive unstable equilibrium; when $0<\beta\ll 1$, $0$ is unstable and there exists a negative and locally asymptotically stable equilibrium.
		\item $a<0,\,b<0$. When $\beta<0$ with $|\beta|\ll 1$, $0$ is unstable; when $0<\beta\ll 1$, $0$ is locally asymptotically stable and there exists a positive unstable equilibrium.
		\item $a>0,\,b<0$. When $\beta<0$ with $|\beta|\ll 1$, $0$ is unstable and there exists a locally asymptotically stable negative equilibrium; when $0<\beta\ll 1$, $0$ is stable and a positive unstable equilibrium appears.
		\item $a<0,\,b>0$. When $\beta$ changes from negative to positive, $0$ changes its stability from stable to unstable. Correspondingly, a negative unstable equilibrium becomes positive and locally asymptotically stable.
	\end{enumerate}
	In particular, if $a>0$ and $b>0$, a backward bifurcation occurs at $\beta=0$, and if $a<0$ and $b>0$, a forward bifurcation occurs at $\beta=0$.
\end{theorem}
\begin{note*}
Remark 1 in~\cite{chavez2004} suggests that if the equilibrium of interest in Theorem~\ref{theo:chavez} is a non-negative equilibrium $x_0$, then the requirement that w is non-negative is not necessary. When some components in w are negative, one can still apply Theorem~\ref{theo:chavez} provided that the $j$-th component of $w$ is positive whenever the $j$-th component of $x_0$ is zero. If the $j$-th component of $x_0$ is positive, then the $j$-th component of $w$ need not be positive. 
\end{note*}
\end{appendices}

\end{document}